\documentclass[11pt,a4paper,oneside,reqno]{amsart}
\usepackage[english]{babel}
\usepackage{amsopn}

\newtheorem{theorem}{Theorem}[section]

\newtheorem{corollary}[theorem]{Corollary}
\newtheorem{lemma}[theorem]{Lemma}
\newtheorem{proposition}[theorem]{Proposition}

\theoremstyle{definition}
\newtheorem*{definition}{Definition}

\newtheorem{example}[theorem]{Example}
\newtheorem{remark}[theorem]{Remark}

\usepackage[square,sort,comma,numbers]{natbib}

\usepackage[all]{xy}
\usepackage{enumerate}
\usepackage{amsfonts,amssymb,amsmath,eucal,pinlabel,array,hhline}
\usepackage{slashed}
\usepackage{tabulary}
\usepackage{fancyhdr}
\usepackage{a4wide}
\usepackage{calrsfs,bbm,dsfont}
\usepackage[position=b]{subcaption}

\newcommand{\Id}{\mathit{Id}}
\newcommand{\id}{\mathit{id}}
 
\usepackage{comment}
\usepackage{amsfonts}
\usepackage{multirow} 
\usepackage{kbordermatrix}
\usepackage{colorinfo}
\usepackage{brunnian} 
\usepackage{braids}
\usetikzlibrary{matrix}
\usetikzlibrary{arrows}

\newcommand{\C}{\mathbb{C}}
\newcommand{\R}{\mathbb{R}}

\newcommand{\Z}{\mathbb{Z}}

\newcommand{\NN}{\mathcal{N}}

\DeclareSymbolFont{EulerScript}{U}{eus}{m}{n}
\DeclareSymbolFontAlphabet\mathscr{EulerScript}
\begin{document}

\title{$L^2$-Burau maps and~$L^2$-Alexander torsions}

\author{Fathi Ben Aribi}
\address{Universit\'e de Gen\`eve, Section de math\'ematiques, 2-4 rue du Li\`evre, 1211 Gen\`eve 4, Switzerland}
\email{fathi.benaribi@unige.ch}

\author{Anthony Conway}
\email{anthony.conway@unige.ch}

\subjclass[2000]{57M25; 57M27} 
\keywords{Braid group, Burau representation, $L^2$-Alexander torsion.}
\begin{abstract}
It is well known that the Burau representation of the braid group can be used to recover the Alexander polynomial of the closure of a braid. We define~$L^2$-Burau maps and use them to compute some~$L^2$-Alexander torsions of links. As an application, we prove that the~$L^2$-Burau maps distinguish more braids than the Burau representation.
\end{abstract}
\maketitle

\section{Introduction}
\label{sec:intro}
The $L^2$-Alexander torsions of $3$-manifolds were introduced by Dubois, Friedl and L\"uck in \cite{DFL} as generalizations of both the classical Reidemeister torsions and the $L^2$-Alexander invariant of knots of Li-Zhang \cite{LZ06a}. These $L^2$-Alexander torsions are topological invariants that are functions on the positive real numbers. On the one hand, the~$L^2$-Alexander torsions share many features with the Alexander polynomial: for instance they are symmetric \cite{DFLdual} and provide information on the Thurston norm of the considered manifold \cite{DFL, FL}. In the classical case, the Alexander polynomial is related to the (reduced) Burau representation of the braid group \cite{Burau}. It is thus natural to ask whether a similar relation exists in the~$L^2$ case. On the other hand, the~$L^2$-Alexander torsions are, in a sense, stronger invariants than their classical counterparts: not only do they contain the simplicial volume of a 3-manifold \cite{Luc02}, they also detect an infinite number of knots \cite{BA13, BAstronger} whereas the Alexander polynomial does not. Therefore, if an~$L^2$-analogue of the Burau representation were to exist, one may expect it to distinguish more braids than the classical Burau representation.

In the present article, we introduce~$L^2$-Burau maps and reduced~$L^2$-Burau maps (see Section \ref{sec:results} for the precise definitions) and study their properties. Although these maps do not provide (anti-)representations of the braid group, they remain computable by recursive formulas and Fox calculus (see Lemma \ref{lem:cocycle} and Proposition \ref{prop:BurauFox}). Moreover, we show that one can extract the classical Burau representation from any $L^2$-Burau map (Proposition \ref{prop:L2classical}). Furthermore, we relate particular $L^2$-Burau maps of braids to~$L^2$-Alexander torsions of the closures of these braids (Theorem \ref{thm:main}). As an application, we provide an example of two braids indistinguishable under the Burau representation but which can be told apart by the~$L^2$ version (Corollary \ref{cor:faithful}). Our main tools rely on well-known results from the theory of $L^2$-invariants together with the homological interpretation of the Burau representation.

\medbreak
The paper is organized as follows. First, in Section \ref{sec:L2preliminaries}, we recall some theory of~$L^2$-invariants. Then, in Section \ref{sec:topPrelim}, we fix notations regarding the braid group and recall the definition of the $L^2$-Alexander torsion together with its relation to Fox calculus. Finally, in Section \ref{sec:results}, we introduce the~$L^2$-Burau maps (Subsections \ref{sub:Burau} and \ref{sub:reduced}) and prove the main results (Subsection \ref{sub:thm}).

\section*{Acknowledgments}
The first author was supported by the \textit{Swiss National Science Foundation}, subsidy \\
$200021\_ 162431$, at the Universit\'e de Gen\`eve. The second author was supported by the NCCR SwissMAP, funded by the \textit{Swiss National Science Foundation}. We thank David Cimasoni for helpful conversations.

\section{Hilbert~$\mathcal{N}(G)$-modules and the~$L^2$-torsion.}
\label{sec:L2preliminaries}

In this section we briefly review some theory of~$L^2$-invariants: we begin with the von Neumann dimension of a finitely generated Hilbert~$\mathcal{N}(G)$-module (Subsection \ref{sub:dimension}) before moving on to the Fuglede-Kadison determinant (Subsection \ref{sub:Fuglede}) and discussing~$L^2$-homology (Subsection \ref{sub:L2chain}). We mostly follow \cite{Luc02} and \cite{DFL}.

\subsection{The von Neumann dimension}
\label{sub:dimension}

Given a countable discrete group~$G$, the completion of the algebra~$\C[G]$ endowed with the scalar product
$ \left \langle \sum_{g \in G} \lambda_g g , \sum_{g \in G} \mu_g g \right \rangle:= \sum_{g \in G} \lambda_g \overline{\mu_g}$
is the Hilbert space ~$$ \ell^2(G):= \left \{ \sum_{g \in G} \lambda_g g \ | \  \lambda_g \in \mathbb{C} , \sum_{g \in G} | \lambda_g |^2 < \infty \right \}$$ of square-summable complex functions on~$G$. We denoted by~$B(\ell^2(G))$ the algebra of  operators on~$\ell^2(G)$ that are bounded with respect to the operator norm. 

Given~$h \in G$, we define the corresponding \textit{left-} and \textit{right-multiplication operators}~$L_{h}$ and~$R_{h}$  in~$B(\ell^2(G))$ as extensions of the automorphisms~$(g \mapsto hg)$ and~$(g \mapsto gh)$ of~$G$.
One can extend the operators~$R_{h}$~$\mathbb{C}$-linearly to operators~$R_w\colon \ell^2(G) \to \ell^2(G)$ for any~$w \in \C[G]$. Moreover, if~$\ell^2(G)^n$ is endowed with its usual Hilbert space structure and~$A = \left ( a_{i,j} \right ) \in M_{p,q}(\C[G])$ is a~$\C[G]$-valued~$p\times q$ matrix, then the right multiplication
$$R_A:= \left (R_{a_{i,j}}\right )_{1 \leqslant i \leqslant p, 1 \leqslant j \leqslant q}$$ provides a bounded operator~$\ell^2(G)^{q} \rightarrow \ell^2(G)^{ p}$. Note that we shall consider elements of~$\ell^2(G)^n$ as \textit{column vectors} and suppose that matrices with coefficients in~$B(\ell^2(G))$ act on the \textit{left} (even though the coefficients may be \textit{right}-multiplication operators).

The \textit{von Neumann algebra}~$\mathcal{N}(G)$ of the group~$G$ is the sub-algebra of~$B(\ell^2(G))$ made up of~$G$-equivariant operators (i.e. operators that commute with all left multiplications~$L_h$). A \textit{finitely generated Hilbert~$\mathcal{N}(G)$-module} consists in a Hilbert space~$V$ together with a left~$G$-action by isometries such that there exists a positive integer~$m$ and an embedding~$\varphi$ of~$V$ into~$\ell^2(G)^m$. A \textit{morphism of finitely generated Hilbert~$\mathcal{N}(G)$-modules}~$f \colon U \rightarrow V$ is a linear bounded map which is~$G$-equivariant.

Denoting by~$e$ the neutral element of~$G$, the von Neuman algebra of~$G$ is endowed with the \textit{trace}~$\mathrm{tr}_{\mathcal{N}(G)} \colon \mathcal{N}(G)  \rightarrow \mathbb{C}, \phi \mapsto  \left \langle \phi (e) , e \right \rangle$ which extends to~$\mathrm{tr}_{\mathcal{N}(G)} \colon M_{n,n}(\mathcal{N}(G)) \rightarrow \C$ by summing up the traces of the diagonal elements.

\begin{definition}
The \textit{von Neumann dimension} of a finitely generated Hilbert~$\mathcal{N}(G)$-module~$V$ is defined as
~$$\dim_{\mathcal{N}(G)}(V) := \mathrm{tr}_{\mathcal{N}(G)}(\mathrm{pr}_{\varphi(V)}) \in \R_{\geqslant 0},$$
where~$ \mathrm{pr}_{\varphi(V)} \colon \ell^2(G)^m \to \ell^2(G)^m~$ is the orthogonal projection onto~$\varphi(V)$.
\end{definition}

 The von Neumann dimension does not depend on the embedding of~$V$ into the finite direct sum of copies of~$\ell^2(G)$.

\subsection{The Fuglede-Kadison determinant}
\label{sub:Fuglede}

The \textit{spectral density}~$F(f) \colon \mathbb{R}_{\geq 0} \to \mathbb{R}_{\geq 0}$ of a morphism~$f \colon U \to V$ of finitely generated Hilbert~$\mathcal{N}(G)$-modules maps~$\lambda \in \mathbb{R}_{\geqslant 0}$ to 
$$ F(f)(\lambda):= \sup \{
 \dim_{\mathcal{N}(G)} (L) | L \in \mathcal{L}(f,\lambda) \},$$
where~$\mathcal{L}(f,\lambda)$ is the set of finitely generated Hilbert~$\mathcal{N}(G)$-submodules of~$U$ on which the restriction of~$f$ has a norm smaller or equal to~$\lambda$. Since~$F(f)(\lambda)$ is monotonous and right-continuous, it defines a measure~$dF(f)$ on the Borel set of~$\R_{\geqslant 0}$ solely determined by the equation~$dF(f)(]a,b]) = F(f)(b)-F(f)(a)$ for all~$a<b$.

\begin{definition} \label{def detFK}
The \textit{Fuglede-Kadison determinant of~$f$} is defined by
\begin{equation*}\label{detFK}
{\det}_{\mathcal{N}(G)}(f)=
\begin{cases}
 \exp \left ( \int_{0^+}^\infty \ln(\lambda) \, dF(f)(\lambda) \right )                       & \mbox{if } \int_{0^+}^\infty \ln(\lambda) \, dF(f)(\lambda) > -\infty, \\ 
0  & \mbox{otherwise }.
 \end{cases} 
\end{equation*}
Moreover, when~$\int_{0^+}^\infty \ln(\lambda) \, dF(f)(\lambda) > -\infty$, one says that \emph{$f$ is of determinant class}. 
\end{definition}

If~$U$ and~$V$ have the same von Neumann dimension, we define the \textit{regular Fuglede-Kadison determinant of~$f$} denote by~${\det}^r_{\mathcal{N}(G)}(f)$ as~${\det}_{\mathcal{N}(G)}(f)$ when~$f$ is injective, and zero otherwise. For later use, let us mention the following property of the determinant (see \cite{Luc02} for the proof).

\begin{proposition} \label{prop operations det}
Let~$G$ be a countable discrete group. If~$g \in G$ is of infinite order, then for all~$t \in \C$ the operator~$ \Id - t R_g$ is injective and~$
\mathrm{det}_{\mathcal{N}(G)} ( \Id - t R_g) = \max ( 1 , |t|)$.
\end{proposition}

\subsection{$L^2$-torsion of a finite Hilbert~$\NN(G)$-chain complex}
\label{sub:L2chain}

A \textit{finite Hilbert~$\NN(G)$-chain complex}~$C_*$ is a sequence of morphisms of finitely generated Hilbert~$\NN(G)$-modules
$$C_* = \left(0 \to C_n \overset{\partial_n}{\longrightarrow} C_{n-1} 
\overset{\partial_{n-1}}{\longrightarrow} \ldots
\overset{\partial_2}{\longrightarrow} C_1 \overset{\partial_1}{\longrightarrow} C_0 \to 0\right)$$
such that~$\partial_p \circ \partial_{p+1} =0$ for all~$p$.
The \textit{$p$-th~$L^2$-homology} of~$C_*$ is the finitely generated Hilbert~$\NN(G)$-module
$$H_p^{(2)}(C_*) := \textrm{Ker}(\partial_p) / \overline{\textrm{Im}(\partial_{p+1})}.$$
The \textit{$p$-th~$L^2$-Betti number of~$C_*$} is defined as~$b_p^{(2)}(C_*) := \dim_{\NN(G)}(H_p^{(2)}(C_*))$. A finite Hilbert~$\NN(G)$-chain complex~$C_*$ is \textit{weakly acyclic} if its~$L^2$-homology is trivial  (i.e. if all its~$L^2$-Betti numbers vanish) and of \textit{determinant class} if all the operators~$\partial_p$ are of determinant class. 

The following result is a reformulation of \cite[Theorem 1.21 and Theorem 3.35 (1)]{Luc02}:

\begin{proposition}\label{prop:exact}
Let~$0 \to 
C_*
\overset{\iota_*}{\longrightarrow} 
D_*
 \overset{\rho_*}{\longrightarrow} 
E_*
  \to 0$ be an exact sequence of finite Hilbert~$\NN(G)$-chain complexes. If two of the finite Hilbert~$\NN(G)$-chain complexes~$C_*, D_*, E_*$ are weakly acyclic (respectively weakly acyclic and of determinant class), then the third is as well.
\end{proposition}

\begin{definition}
The \textit{$L^2$-torsion} of a finite Hilbert~$\NN(G)$-chain complex~$C_*$ is defined as
$$T^{(2)}(C_*) := \prod_{i=1}^n \det {}_{\NN(G)}(\partial_i)^{(-1)^i} \in \R_{>0}$$
when $C_*$ is  weakly acyclic and of determinant class, 
and as~$T^{(2)}(C_*):=0$ otherwise.
\end{definition}

Let~$C_*=(0 \to \ell^2(G)^k \overset{\partial_2}{\longrightarrow}  \ell^2(G)^{k+l} \overset{\partial_1} {\longrightarrow}  \ell^2(G)^l \to 0)$ be a finite Hilbert~$\NN(G)$-chain complex and let~$J \subset \{1, \ldots,k+l\}$ be a subset of size~$l$. Viewing~$\partial_1,\partial_2$ as matrices with coefficients in~$B(\ell^2(G))$, denote  by~$\partial_1(J)$ the operator composed of the columns of~$\partial_1$ indexed by~$J$, and by~$\partial_2(J)$ the operator obtained from~$\partial_2$ by deleting the rows indexed by~$J$. We refer to \cite[Lemma 3.1]{DFL} for the proof of the following proposition.

\begin{proposition}
\label{prop:tau chaine}
Assume that~$\partial_1(J)$ is injective and of determinant class.
Then~$$T^{(2)}(C_*) = \dfrac{\det_{\NN(G)}^r(\partial_2(J))}{\det_{\NN(G)}^r(\partial_1(J))}.$$ In particular,~$\partial_2(J)$ is injective and of determinant class if and only if~$C_*$ is weakly acyclic and of determinant class, and in this case one can write
$$T^{(2)}(C_*)  = \dfrac{\det_{\NN(G)}(\partial_2)}{\det_{\NN(G)}(\partial_1)} = \dfrac{\det_{\NN(G)}(\partial_2(J))}{\det_{\NN(G)}(\partial_1(J))}.$$
\end{proposition}

\section{Topological preliminaries}
\label{sec:topPrelim}

In this section, we start by reviewing the braid group (Subsection \ref{sub:braids}), before discussing the~$L^2$-homology of CW-complexes (Subsection \ref{sub:L2homology}) and the~$L^2$-Alexander torsions together with their relation to Fox calculus (Subsection \ref{sub:L2TorsionFox}).

\begin{figure}[h]
\labellist\small\hair 2.5pt
\pinlabel {$\beta_1$} at 1 760
\pinlabel {$\beta_2$} at 1 615
\pinlabel {$\beta_1 \beta_2$} at 390 685
\endlabellist
\centering
\includegraphics[width=0.4\textwidth,scale=1.2]{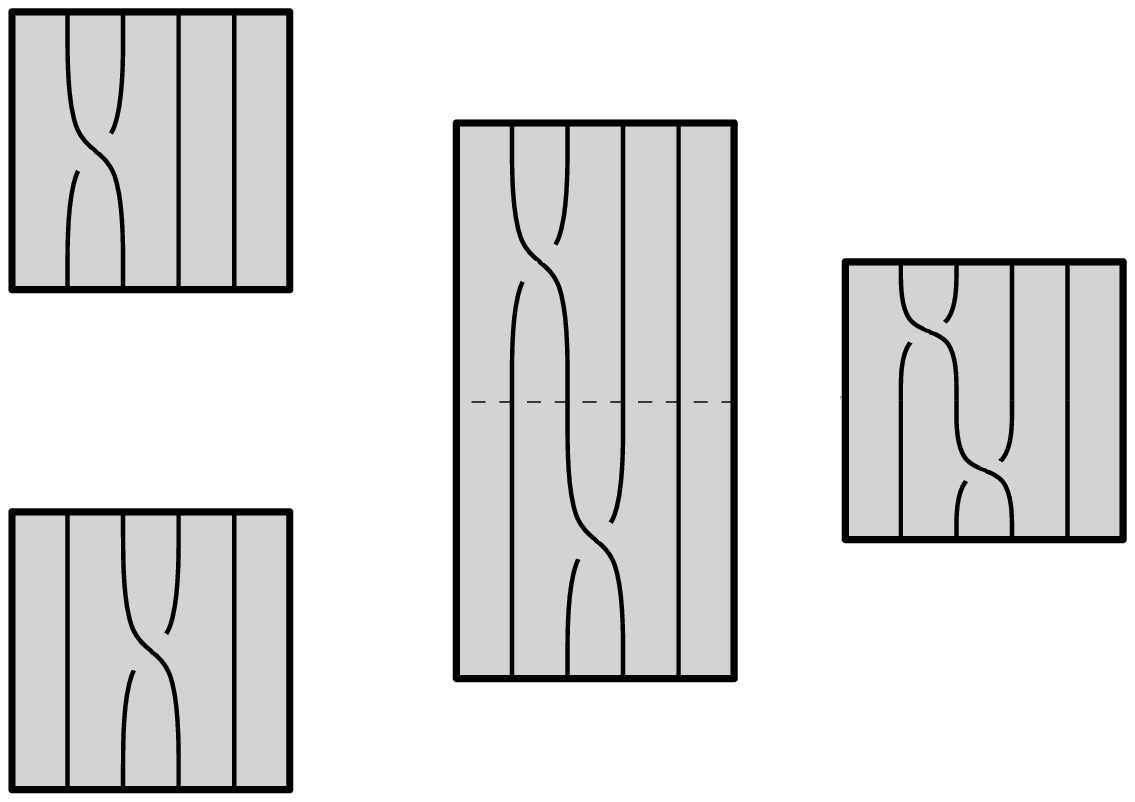}
\caption{Two braids $\beta_1, \beta_2$ and their composition, the~braid~$\beta_1\beta_2$.}
\label{fig:CompositionTwistedOriente}
\end{figure}

\subsection{The braid group}
\label{sub:braids}
Following Birman \citep{Birmanbook}, we start by recalling some well-known properties of the braid group $B_n$ including its \textit{right} action on the free group $F_n$. In contrast with Birman however, the composition of maps will be written in the usual way (from right to left) which leads to an \textit{anti}-representation $B_n \rightarrow Aut(F_n)$. 
\medbreak
Let~$D^2$ be the closed unit disk in~$\mathbb{R}^2.$ Fix a set of~$n \geq 1$ punctures~$p_1,p_2,\dots,p_n$ in the interior of~$D^2$. We shall assume that each~$p_i$ lies in~$(-1,1)=Int(D^2) \cap \mathbb{R}$ and~$p_1<p_2<\dots<p_n.$ A \textit{braid with~$n$ strands} is an~$n$-component piecewise linear one-dimensional submanifold~$\beta$ of the cylinder~$D^2 \times [0,1]$ whose boundary is~$\bigsqcup_{i=1}^n p_i \times \lbrace 0,1 \rbrace$, and where the projection to~$[0,1]$ maps each component of~$\beta$ homeomorphically onto~$[0,1]$. Two braids~$\beta_1$ and~$\beta_2$ are \textit{isotopic} if there is a self-homeomorphism~$H$ of~$D^2 \times [0,1]$ which keeps~$D^2 \times \lbrace 0,1 \rbrace \cup \partial D^2 \times [0,1]$ fixed and such that~$H(\beta_1)=\beta_2$. The \textit{braid group}~$B_n$ consists of the set of isotopy classes of braids.  The identity element is given by the \textit{trivial braid}~$\xi_n:=\lbrace p_1,p_2,\dots, p_n \rbrace \times [0,1]$ while the composition~$\beta_1 \beta_2$ consists in gluing~$\beta_1$ on top of~$\beta_2$ and shrinking the result by a factor~$2$ (see Figure \ref{fig:CompositionTwistedOriente}). 

The braid group~$B_n$ can also be seen as the set of  isotopy classes of orientation-preserving homeomorphisms of~$D_n :=D^2 \setminus \lbrace p_1,\dots, p_n \rbrace$ fixing the boundary pointwise. Either way,~$B_n$ admits a presentation with~$n-1$ generators~$\sigma_1,\sigma_2, \dots, \sigma_{n-1}$ subject to the relations~$\sigma_i \sigma_{i+1} \sigma_i=\sigma_{i+1} \sigma_i \sigma_{i+1}$ for each~$i$, and~$\sigma_i \sigma_j = \sigma_j \sigma_i$ if~$|i-j|>2$. Topologically, the generator~$\sigma_i$ is the braid whose~$i$-th component passes over the~$i+1$-th component.

\begin{figure}[h]
\labellist\small\hair 2.5pt
\pinlabel {$x_1$} at 35 745
\pinlabel {$x_2$} at 80 745
\pinlabel {$x_3$} at 125 745
\pinlabel {$z$} at 80 839
\endlabellist
\centering
\includegraphics[width=0.4\textwidth,scale=0.6]{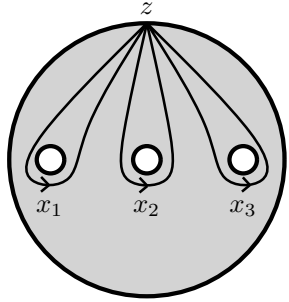}
\caption{The punctured disk~$D_3$. }
\label{fig:DiskTwisted}
\end{figure}

Fix a base point~$z$ of~$D_n$ and denote by~$x_i$ the simple loop based at~$z$ turning once around~$p_i$ counterclockwise for~$i=1,2,\dots, n$ (see Figure \ref{fig:DiskTwisted}). The group~$\pi_1(D_n)$ can then be identified with the free group~$F_n$ on the~$x_i.$ 
If~$H_\beta$ is a homeomorphism of~$D_n$ representing a braid~$\beta$, then the induced automorphism~$h_{\beta}$ of the free group~$F_n$  depends only on~$\beta$. It follows from the way we compose braids that~$h_{\alpha \beta}=h_{\beta}\circ h_{\alpha}$, and the resulting \textit{right} action of~$B_n$ on~$F_n$ can be explicitly described by
\[
h_{\sigma_i}(x_j)=
\begin{cases}
x_i x_{i+1} x_i^{-1}  & \mbox{if }  j=i, \\
 x_i                            & \mbox{if }  j=i+1, \\ 
x_j                             & \mbox{otherwise. } \\ 
 \end{cases} 
\] 
The \textit{closure} of a braid~$\beta \in B_n$ is the oriented link~$\hat{\beta}$ in the three-sphere obtained from~$\beta$ by adding~$n$ parallel strands in~$S^3 \setminus ({D^2 \times [0,1]})$ (see Figure \ref{fig:Closure}).

 \begin{figure}[h]
\labellist\small\hair 2.5pt
\pinlabel {$\beta$} at 60 720
\pinlabel {$\hat{\beta}$} at 350 720
\endlabellist
\centering
\includegraphics[width=0.35\textwidth,scale=1.2]{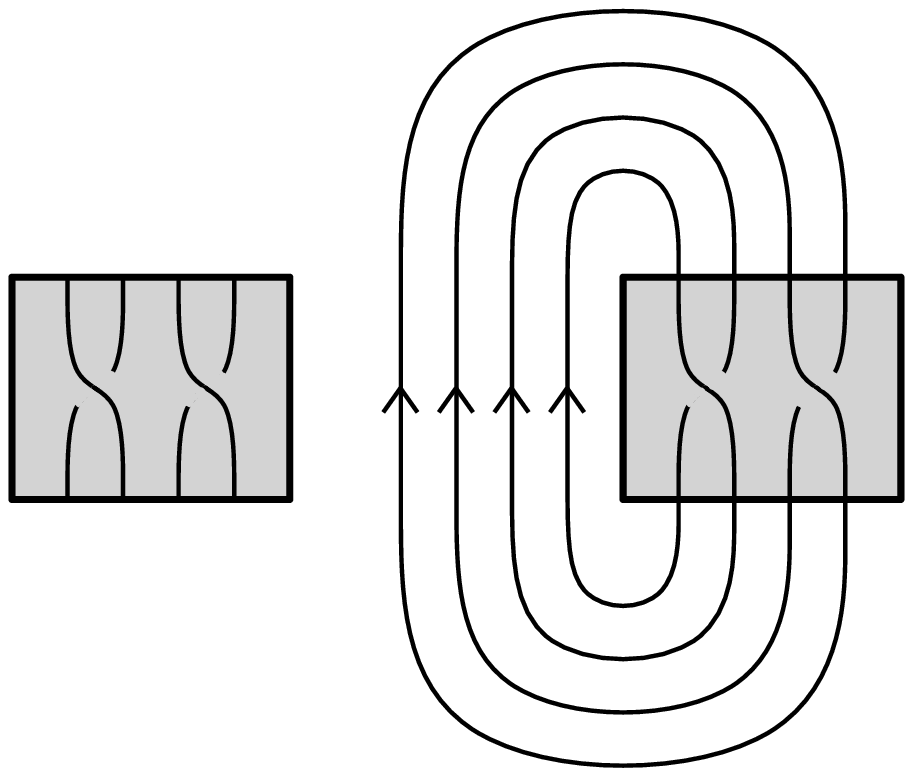}
\caption{The closure of a braid.}
\label{fig:Closure}
\end{figure}

\subsection{$L^2$-homology of CW-complexes}
\label{sub:L2homology}
Following \cite{Luc02,DFL}, we recall the definition of the $L^2$-homology of a CW-complex associated with an admissible triple. We then make an explicit computation in the case of the punctured disk.
\medbreak
Let~$X$ be a compact connected CW-complex endowed with a basepoint~$z$, and let~$Y$ be a connected CW-subcomplex of~$X$. We denote by~$p\colon \widetilde{X} \rightarrow X$ the universal cover of~$X$ and write~$\widetilde{Y}=p^{-1}(Y)$. 
Setting~$\pi = \pi_1(X,z)$, an \textit{admissible triple}~$(\pi,\phi,\gamma)$ consists in homomorphisms~$\phi \colon \pi \rightarrow \mathbb{Z}$ and~$\gamma \colon \pi \rightarrow G$ such that~$\phi$ factors through~$\gamma$. Given such a triple and $t>0$, if we denote by
$$\kappa(\pi, \phi, \gamma,t) \colon \mathbb{Z}[\pi] \rightarrow \mathbb{R}[G]$$ 
the ring homomorphism determined by~$g \mapsto t^{\phi(g)}\gamma(g)$ for $g \in \pi$, then there is a right action of~$\pi$ on~$\ell^2(G)$ given by~$a \cdot g= R_{\kappa(\pi, \phi, \gamma,t)(g)} (a)$, where~$a \in \ell^2(G)$ and~$g \in \pi$; this turns~$\ell^2(G)$ into a right~$\mathbb{Z}[\pi]$-module.

On the other hand, the natural left action of~$\pi = \pi_1(X,z)$ on~$\widetilde{X}$ gives rise to a left~$\mathbb{Z}[\pi]$-module structure on the cellular chain complex~$C_*(\widetilde{X},\widetilde{Y})$. The \textit{$\mathcal{N}(G)$-cellular chain complex} of the pair~$(X,Y)$ associated to~$(\phi,\gamma,t)$ is the finite Hilbert~$\NN(G)$-chain complex
$$ C_*^{(2)}(X,Y; \phi, \gamma,t)=\ell^2(G) \otimes_{\mathbb{Z}[\pi]}C_*(\widetilde{X},\widetilde{Y}),$$
and the \textit{$L^2$-homology} of~$(X,Y)$ associated to~$(\phi,\gamma,t)$, denoted ~$H_*^{(2)}(X,Y;\phi, \gamma,t)$, is obtained by taking the~$L^2$-homology of~$C_*^{(2)}(X,Y;\phi, \gamma,t)$.

\begin{lemma}
\label{lem:Relative$L^2$-}
Given~$z \in D_n$, for all admissible $(\pi,\phi,\gamma)$ and all $t>0$, the finitely generated Hilbert~$\mathcal{N}(G)$-module~$H_1^{(2)}(D_n,z;\phi, \gamma,t)$ has von Neumann dimension $n$.
\end{lemma}

\begin{proof}
The punctured disk~$D_n$ is simple homotopy equivalent to~$X$, the wedge of the~$n$ loops representing the generators of~$\pi_1(D_n)$ described in Section \ref{sub:braids}. As a consequence, it follows from \cite[Theorem 1.21]{Luc02} and the proof of \cite[Theorem 2.12]{BAthesis} that $H_1^{(2)}(X,z;\phi, \gamma,t)$ and $H_1^{(2)}(D_n,z;\phi, \gamma,t)$ have same von Neumann dimension. Thus it suffices to prove the claim for $X$.

 Choose a cellular decomposition of this latter space~$X$ consisting of the~$0$-cell~$z$ (the basepoint of the wedge) and one~$1$-cell~$x_i$ for each loop. For~$i=1,2,\dots,n$, let~$\widetilde{x}_i$ be the lift of~$x_i$ starting at an (arbitrary) fixed lift of~$z$. With this cell structure, the~$\mathcal{N}(G)$-cellular chain complex of the pair~$(X,z)$ associated to~$(\phi,\gamma,t)$ is~$ 0 \to C_1^{(2)}(X,z;\phi, \gamma,t) \to C_0^{(2)}(X,z;\phi, \gamma,t) \to 0$, where
$$ C_1^{(2)}(X,z;\phi, \gamma,t)
 =\ell^2(G) \otimes_{\mathbb{Z}[\pi]} C_1(\widetilde{X},\widetilde{z}) \\
\cong \bigoplus_{i=1}^n \ell^2(G) \widetilde{x}_i.~$$
Since~$C_0^{(2)}(\widetilde{X},\widetilde{z})$ vanishes,~$H_1^{(2)}(X,z;\phi, \gamma,t)=C_1^{(2)}(X,z;\phi, \gamma,t)$ and the claim follows. 
\end{proof}

Given an admissible triple~$\left (\pi_1(X',z'),\phi,\gamma\colon \pi_1(X',z') \to G\right)$, note that a basepoint-preserving homeomorphism of pairs~$F \colon (X,Y) \rightarrow (X',Y')$ induces an isomorphism~$f\colon \pi_1(X) \to \pi_1(X')$ and 
isomorphisms of finitely generated Hilbert~$\mathcal{N}(G)$-modules
$$H_i^{(2)}(F)\colon
H_i^{(2)}(X,Y; \phi \circ f, \gamma \circ f,t) 
\to
H_i^{(2)}(X',Y';\phi, \gamma, t).$$
The precomposition by~$f$ is required for the homeomorphism~$F$ to induce a well-defined chain map.

\begin{example}
\label{ex:braidInduced}
Fix a basepoint~$z \in \partial D_n$ as in Figure \ref{fig:DiskTwisted}. Let~$H_\beta: D_n \rightarrow D_n$ be a homeomorphism representing a braid~$\beta \in B_n$. As~$H_\beta$ fixes the boundary of the disk, it lifts uniquely to a homeomorphism~$\widetilde{H}_\beta\colon \widetilde{D}_n \rightarrow \widetilde{D}_n$ which preserves a fixed lift of~$z$. Up to homotopy, this lift depends uniquely on the isotopy class of~$H_\beta$ and consequently the map induced on the chain group~$C_1(\widetilde{D}_n,\widetilde{z})$ depends uniquely on the braid~$\beta$. 

Denote by~$\phi  \colon \pi_1(D_n) \rightarrow \mathbb{Z}$ the epimorphism defined by~$x_i \mapsto 1$. Fixing~$t>0$ and a homomorphism~$\gamma\colon \pi_1(D_n) \to G$ through which~$\phi$ factors, each braid~$\beta$ induces a well-defined isomorphism of finitely generated Hilbert~$\mathcal{N}(G)$-modules:
$$  H_1^{(2)}(H_{\beta})\colon
H_1^{(2)}(D_n,z;\phi \circ h_{\beta},\gamma \circ h_{\beta},t) 
\to H_1^{(2)}(D_n,z;\phi ,\gamma,t).$$
Since $\phi \circ h_{\beta} = \phi$ for all braids $\beta \in B_n$, from now on we shall write $\phi$ instead of $\phi \circ h_{\beta}$.
\end{example}

\subsection{$L^2$-Alexander torsion and Fox calculus}
\label{sub:L2TorsionFox}
Following \cite{DFL}, we define the $L^2$-Alexander torsion and outline its relation to Fox calculus. We also recall the definition of the $L^2$-Alexander torsion associated to a link and discuss its behavior on split links. 
\medbreak
Given a CW-complex~$X$, fix an admissible triple~$(\pi_1(X),\phi, \gamma)$.
\begin{definition}
The \textit{$L^2$-Alexander torsion} of~$(X,\phi,\gamma)$ at~$t>0$ is defined as
$$ T^{(2)}(X,\phi,\gamma)(t) := T^{(2)}\left (C_*^{(2)}(X;\phi,\gamma,t)\right ).$$
Observe that~$T^{(2)}(X,\phi,\gamma)(t) \neq 0$ if and only if $C_*^{(2)}(X;\phi,\gamma,t)$ is weakly acyclic and of determinant class.
\end{definition}

Note that~$L^2$-Alexander torsions are only defined up to multiplication by $(t \mapsto t^k)$ with $k \in \mathbb{Z}$. For this reason, we shall write $f(t) \ \dot{=} \ g(t)$ if $f$ is equal to $g$ up to multiplication by $(t \mapsto t^k)$ for $k \in \mathbb{Z}$. Moreover the~$L^2$-Alexander torsions are invariant by simple homotopy equivalence \cite{chapman,Luc02,DFL,BAthesis}. Using this fact, we briefly review Fox calculus and outline how it can be used to compute the~$L^2$-Alexander torsion.

Denoting by~$F_n$ the free group on~$x_1,x_2,\dots,x_n$, the \textit{Fox derivative} (first introduced by Fox \citep{Fox})
$\dfrac{\partial}{\partial x_i}: \mathbb{Z}[F_n] \rightarrow \mathbb{Z}[F_n]~$  is the linear extension of the map defined on elements of~$F_n$ by 
$$
\frac{\partial x_j}{\partial x_i}=\delta_{ij}, \ \ \ \ \ \ \ \ \ \frac{\partial x_j^{-1}}{\partial x_i}=-\delta_{ij} x_j^{-1}, \ \ \ \ \ \ \ \ \  \frac{\partial (uv)}{\partial x_i}=\frac{\partial u}{\partial x_i}+u\frac{\partial v}{\partial x_i}.
$$
If~$P = \langle x_1, \ldots, x_n | r_1, \ldots, r_{m} \rangle$ is a presentation of a group~$\pi$, construct the~$2$-complex~$W_P$ with one~$0$-cell~$v$,~$n$ oriented~$1$-cells labeled~$x_1, x_2,\dots,x_n$ and~$m$ oriented~$2$-cells~$c_1, c_2,\dots, c_{m}$ with each~$\partial c_j$ glued to the~$1$-cells according to the word~$r_i$. Note that $\pi_1(W_P)=\pi$ and let~$\widetilde{v}, \widetilde{x}_i$ and~$\widetilde{c}_j$ be corresponding lifts to the universal cover~$p: \widetilde{W}_P \rightarrow W_P$ (i.e. each~$\widetilde{x}_i$ starts at~$\widetilde{v}$ and the first word in the boundary of~$\widetilde{c}_j$ is of the form~$\widetilde{x}_i$).

Denote by~$\mathrm{pr}: \mathbb{Z}[F_n] \rightarrow \mathbb{Z}[\pi]$ the ring homomorphism induced by the quotient map. The ~$\Z[\pi]$-module~$C_1(\widetilde{W}_P,p^{-1}(v))$ is generated by the~$\widetilde{x}_i$, and if~$w$ is a word in the~$x_i$, then its lift~$\widetilde{w}$ (viewed as a~$1$-chain in the universal cover) can be written as
$$ \widetilde{w}=\sum_{i=1}^n \mathrm{pr} \left( \frac{\partial w}{\partial
x_i} \right) \widetilde{x}_i.$$
Since the boundary map~$\partial_2$ of the chain complex~$C_*(\widetilde{W}_P)$ sends~$\widetilde{c}_j$ to the lift of~$r_j$ beginning at~$\widetilde{v}$, the previous equation specializes to
$$ \partial_2(\widetilde{c}_j)=\sum_{i=1}^n  \mathrm{pr} \left( \frac{\partial r_j}{\partial x_i} \right) \widetilde{x}_i.$$
We shall assume that the elements in the chain complex~$C_*(\widetilde{W}_P)$ of free left~$\mathbb{Z}[\pi]$-modules are column vectors and that the matrices of the differentials act by left multiplication.  Consequently,~$\partial_2$ is represented by the~$(n \times m)$ matrix whose~$(i,j)$-coefficient is 
$ \mathrm{pr} \left( \frac{\partial r_j}{\partial
x_i} \right).~$

Combining these remarks with Propositions \ref{prop operations det}, \ref{prop:tau chaine} and the fact that
for any integer~$k$ and $t>0$,~$\max\left (1,t^k\right ) = t^{\frac{k-|k|}{2}} \max(1,t)^{|k|} \ \dot{=} \ \max(1,t)^{|k|}$,
 the following result is immediate.

\begin{proposition}
\label{prop:Fox torsion}
Let~$P= \langle x_1, \ldots, x_n | r_1, \ldots, r_{n-1} \rangle$ be a deficiency one presentation of a group~$\pi$, fix $t>0$ and let~$(\pi, \phi\colon \pi \to \Z,\gamma\colon \pi \to G)$ be an admissible triple.
If one denotes by~$A$ the~matrix
in~$M_{n-1,n-1}(\C[G])$
 whose~$(i,j)$ component is
$$ \kappa(\pi,\phi,\gamma,t)\left (
\mathrm{pr} \left (
\dfrac{\partial r_j}{\partial x_i} \right )\right )$$ and one assumes that~$\gamma(x_n)$ has infinite order in~$G$, then
$$T^{(2)}(W_P,\phi,\gamma)(t)
\ \dot{=} \ \dfrac{
\det^r_{\mathcal{N}(G)}\left (
R_{A} \right )}
{ \det^r_{\mathcal{N}(G)}\left (t^{\phi(x_n)} R_{\gamma(x_n)} - \Id\right )} 
\ \dot{=} \  \dfrac{
\det^r_{\mathcal{N}(G)}\left (
R_{A} \right )}
{ \max(1,t)^{|\phi(x_n)|}}.
$$
Moreover, if $M$ is an irreducible~$3$-manifold with non-empty toroidal boundary and infinite $\pi = \pi_1(M)$, then
$$T^{(2)}(M,\phi,\gamma)(t)
\ \dot{=} \ \dfrac{
\det^r_{\mathcal{N}(G)}\left (
R_{A} \right )}
{ \det^r_{\mathcal{N}(G)}(t^{\phi(x_n)} R_{\gamma(x_n)} - \Id)} 
\ \dot{=} \  \dfrac{
\det^r_{\mathcal{N}(G)}\left (
R_{A} \right )}
{\max(1,t)^{|\phi(x_n)|}} .
$$
\end{proposition}

The second part of the above proposition uses the fact that the~$L^2$-Alexander torsions are invariant under simple homotopy equivalence and the following Lemma \ref{lem:simplehomot}.
Although the content of this lemma is known \cite[Section 3.2]{aschenbrenner20123}, we present a short proof that (somewhat appropriately) uses $L^2$-Betti numbers.

\begin{lemma}
\label{lem:simplehomot}
Let~$M$ be an irreducible~$3$-manifold with non-empty toroidal boundary and infinite fundamental group. If~$P$ is a deficiency one presentation of~$\pi_1(M)$, then~$M$ is simple homotopy equivalent to~$W_P$.
\end{lemma}

\begin{proof}
As~$M$ is an irreducible~$3$-manifold with infinite fundamental group, it is aspherical \cite[Paragraph C.1]{aschenbrenner20123}. Since the Whitehead group of the fundamental group of a compact, orientable, non-spherical irreducible 3-manifold is trivial \cite[Paragraph C.36]{aschenbrenner20123}, one only needs to show that~$M$ and~$W_P$ are homotopy equivalent. Consequently it remains to prove that ~$W_P$ is aspherical: indeed both spaces would then be~$K(\pi_1(M),1)$'s. The first~$L^2$-Betti number of a finite CW-complex depends only on its fundamental group \cite[Section 2.2]{Hill02}, therefore~$b^{(2)}_1(W_P) = b^{(2)}_1(\pi_1(M)) =  b^{(2)}_1(M)$. Since~$M$ is prime and has infinite fundamental group,~$b^{(2)}_1(M)=0$ by \cite[Theorem 4.1]{Luc02} and therefore~$b^{(2)}_1(W_P) = 0$. As~$P$ has deficiency one,~$\chi(W_P)=0$, and so~$b^{(2)}_1(W_P)=\chi(W_P)$. It now follows that~$W_P$ is aspherical  \cite[Theorem 2.4]{Hill02}.
\end{proof}

%
%
Given an oriented link~$L = L_1 \cup \ldots \cup L_\mu$ in~$S^3$, denote by~$M_L$ its exterior and by~$G_L=\pi_1(M_L)$ its group. Since any homomorphism~$\phi \colon G_L \to \Z$ factors through the abelianization~$\alpha_L \colon G_L \to H_1(M_L) \cong \mathbb{Z}^\mu$, it is determined by integers~$n_1,\ldots,n_\mu$. Following the notation of \citep{BAthesis}, we denote by~$(n_1,\dots,n_\mu) \colon H_1(M_L) \rightarrow \mathbb{Z}$ the map sending the~$i$-th meridian of~$L$ to~$n_i$ (thus~$\phi = (n_1, \ldots, n_{\mu}) \circ \alpha_L$) and we call
$$T^{(2)}_{L,(n_1, \ldots,n_\mu)}(\gamma)(t)
:= T^{(2)}(M_L, \phi,\gamma)(t)$$
the \textit{$L^2$-Alexander torsion associated to the link~$L$} and the morphism~$\gamma\colon G_L \to G$ at the value~$t>0$.

Although the next lemma is certainly well-known to the experts, to the best of our knowledge, it does not appear in this form in the literature. Therefore, we include it in our discussion.

\begin{lemma}
\label{lem:split}
Let $L$ be a link  and $n_1, \ldots, n_{\mu} \in \Z$. The following assertions are equivalent:
\begin{enumerate}
\item $L$ is split.
\item For all $t>0$, $C^{(2)}(M_L,(n_1, \ldots,n_\mu) \circ \alpha_L, \id,t)$ is not weakly acyclic.
\item The~$L^2$-Alexander torsion~$\left (t \mapsto T^{(2)}_{L,(n_1, \ldots,n_\mu)}(\id)(t) \right )$ vanishes everywhere.
\end{enumerate}
\end{lemma}

\begin{proof}
If the~$\mu$-component link~$L$ is not split, then its exterior~$M_L$ is irreducible, and it follows from \cite{L} that for all integers~$n_1, \ldots, n_{\mu}$ and all~$t>0$,~$C^{(2)}(M_L,(n_1, \ldots,n_\mu) \circ \alpha_L, \id,t)$ is weakly acyclic and of determinant class. Thus, in this case~$T^{(2)}_{L,(n_1, \ldots,n_\mu)}(\id)(t)$ is non-zero, proving that $(3) \Rightarrow (1)$. Moreover, $(2) \Rightarrow (3)$ is immediate.
 
Let us prove $(1) \Rightarrow (2)$. If~$L$ is split, then~$M_L$ is not irreducible, and one can write
$M_L = M_1 \sharp \ldots \sharp M_{r}$, where the~$M_i$ are irreducible link exteriors in~$S^3$. Let us order the~$M_i$ so that
$$M_{L} = (M_1 \setminus B^3) \cup
\left (\bigcup_{i=2}^{r-1} (M_i \setminus (B^3 \sqcup B^3))\right ) \cup (M_r \setminus B^3),$$ where the intersection is a disjoint union of~$r-1$ spheres~$S^2$.
Fix~$t>0$, integers~$n_1, \ldots,n_\mu \in \Z$ (we denote~$\phi_L = (n_1, \ldots,n_\mu) \circ \alpha_L$) and let~$j_i$ be the group monomorphism induced by the inclusion of~$M_i$ minus one or two balls into~$M_L$.
An immediate generalization of the proof of \cite[Theorem 3.1]{BAthesis} (see also \cite{Luc02}) implies that
$$0 \to \bigoplus_{i=1}^{r-1} C_{*}^{(2)}(S^2,1,1,t) \to
\begin{matrix}
C_{*}^{(2)}(M_1 \setminus B^3,\phi_L \circ j_1, j_1,t) \oplus	\\
\bigoplus_{i=2}^{r-1} C_{*}^{(2)}(M_i \setminus (B^3 \sqcup B^3),\phi_L \circ j_i, j_i,t) \\
\oplus C_{*}^{(2)}(M_r \setminus B^3,\phi_L \circ j_r, j_r,t)
\end{matrix}
\to
C_{*}^{(2)}(M_L,\phi_L,\id,t) \to 0
$$
is an exact sequence of finite Hilbert~$\NN(G_L)$-chain complexes.

Now, for all~$i=1, \ldots, r-1$, we add a term~$\ell^2(G_L) \widetilde{B^3} \oplus \ell^2(G_L) \widetilde{B^3}$ to the~$i$-th summand of the left part of the sequence and to the~$i$-th and~$(i+1)$-th summands of the middle part (one ball for each), where the boundary operators send one~$\widetilde{B^3}$ to the corresponding~$\widetilde{S^2}$ and the other to a corresponding~$-\widetilde{S^2}$. Since this process does not change exactness of the sequence, it follows that
$$0 \to 
\bigoplus_{i=1}^{r-1} C_{*}^{(2)}(S^3,1,1,t) \to
\bigoplus_{i=1}^{r} C_{*}^{(2)}(M_i,\phi_L \circ j_i, j_i,t) 
\to
C_{*}^{(2)}(M_L,\phi_L,\id,t) \to 0
$$
remains an exact sequence of finite Hilbert~$\NN(G_L)$-chain complexes.
Each $j_i$ is an injective group homomorphism and thus induces an induction functor, as explained in \cite[Section 1.1.5]{Luc02}. 
As weak acyclicity is unaffected by these induction functors,
the first part of the proof applied to the irreducible pieces $M_i$ shows that~$\bigoplus_{i=1}^{r} C_{*}^{(2)}(M_i,\phi_L \circ j_i, j_i,t)$ is weakly acyclic. Since the left part of the above short exact sequence is not weakly acyclic (see \cite[Theorem 1.35 (8)]{Luc02}), neither is~$C^{(2)}(M_L,(n_1, \ldots,n_\mu) \circ \alpha_L, \id,t)$ (by Proposition \ref{prop:exact}).  
\end{proof}

\section{The~$L^2$-Burau maps and the~$L^2$-Alexander torsions}
\label{sec:results}

In this section, we define the~$L^2$-Burau maps (Subsection \ref{sub:Burau}), the reduced~$L^2$-Burau maps (Subsection \ref{sub:reduced}) and relate the latter to some~$L^2$-Alexander torsions of links (Subsection \ref{sub:thm}).

\subsection{The~$L^2$-Burau map}
\label{sub:Burau}
In this subsection we define $L^2$-Burau maps and show how to compute them using Fox calculus. We wish to emphasize that since our conventions differ from \cite{Conway} (see Section \ref{sub:braids}), the resulting maps nearly behave as \textit{anti}-representations (instead of representations).
\medbreak
Denote by~$\phi  \colon \pi_1(D_n) \rightarrow \mathbb{Z}$ the epimorphism defined by~$x_i \mapsto 1$. Fix~$t>0$ and a homomorphism~$\gamma\colon \pi_1(D_n) \to G$ through which~$\phi$ factors. Given a basepoint~$z \in \partial D_n$, we saw in Example \ref{ex:braidInduced} that each braid~$\beta \in B_n$ induces a well-defined isomorphism of finitely generated~$\mathcal{N}(G)$-modules~
$$H^{(2)}_1(H_{\beta}) \colon
 H_1^{(2)}(D_n,z;\phi,\gamma \circ h_{\beta},t)
  \longrightarrow 
  H_1^{(2)}(D_n,z;\phi,\gamma,t) .
$$
 Using the same notations as in the proof of Lemma \ref{lem:Relative$L^2$-}, we shall call the basis resulting from the isomorphism
$$H_1^{(2)}(D_n,z;\phi,\gamma,t) \cong \bigoplus_{i=1}^n \ell^2(G) \widetilde{x}_i$$
the \textit{good basis} of~$H_1^{(2)}(D_n,z;\phi,\gamma,t).$ With respect to the good bases of~$H_1^{(2)}(D_n,z;\phi,\gamma \circ h_{\beta},t)~$ and~$H_1^{(2)}(D_n,z;\phi,\gamma,t)$, the isomorphism of finitely generated~$\mathcal{N}(G)$-modules~$H^{(2)}_1(H_{\beta})$ gives rise to a~$n \times n$ matrix $\mathcal{B}^{(2)}_{t,\gamma}(\beta)$ with coefficients in~$B(\ell^2(G))$.

\begin{definition}
The \textit{$L^2$-Burau map}~$\mathcal{B}^{(2)}_{t,\gamma}$
associated to the value~$t>0$ and the homomorphism~$\gamma$ sends a braid~$\beta \in B_n$ to the matrix
$\mathcal{B}^{(2)}_{t,\gamma}(\beta)
\in M_{n,n}\left (B(\ell^2(G))\right )$ representing the
isomorphism of finitely generated Hilbert~$\mathcal{N}(G)$-modules defined above. 
\end{definition}

The next lemma shows that while the~$L^2$-Burau map is generally not an (anti-)representation, it is nevertheless determined by the generators of~$B_n$.

\begin{lemma}
\label{lem:cocycle}
Given two braids~$\alpha,\beta \in B_n$, the equation
$$
\mathcal{B}^{(2)}_{t,\gamma}(\alpha \beta)
=
\mathcal{B}^{(2)}_{t,\gamma}(\beta) \circ
\mathcal{B}^{(2)}_{t,\gamma \circ h_{\beta}}(\alpha)
$$
holds for all~$t>0$ and for all~$\gamma \colon \pi_1(D_n) \rightarrow G$ through which~$\phi$ factors.
\end{lemma}

\begin{proof} Since the lift of $H_{\alpha \beta}$ to the universal cover coincides with the lift of $H_{\beta}\circ H_{\alpha}$, the composition 
$$  H_1^{(2)}(D_n,z;\phi,\gamma \circ h_{\alpha  \beta},t) 
\stackrel{\mathcal{B}^{(2)}_{t,\gamma \circ h_{\beta}}(\alpha)}{\longrightarrow} 
H_1^{(2)}(D_n,z;\phi,\gamma \circ h_{\beta},t)
\stackrel{\mathcal{B}^{(2)}_{t,\gamma}(\beta)}{\longrightarrow} 
H_1^{(2)}(D_n,z;\phi,\gamma,t)
$$
coincides with the map~$\mathcal{B}^{(2)}_{t,\gamma}(\alpha \beta)$. 
\end{proof}

In particular, Lemma \ref{lem:cocycle} shows that if one picks a homomorphism~$\gamma$ satisfying~$\gamma \circ h_\beta=\gamma$ for each~$\beta \in B_n$, then the~$L^2$-Burau maps~$\mathcal{B}^{(2)}_{t,\gamma}$ yield anti-representations of the braid group. More generally, fixing~$\gamma\colon \pi_1(D_n) \rightarrow G$, the~$L^2$-Burau maps~$\mathcal{B}^{(2)}_{t,\gamma}$ provide anti-representations of~$B_n^\gamma:= \lbrace \beta \in B_n \ | \ \gamma \circ h_\beta=\gamma \rbrace$.

The next proposition shows that the $L^2$-Burau map can be computed via Fox calculus.
 
\begin{proposition}
\label{prop:BurauFox}
Let~$\beta \in B_n$ be a braid. If one denotes by~$A$ the~$(n \times n)$-matrix whose~$(i,j)$ component is~$$
\kappa(\pi_1(D_n),\phi,\gamma,t)
  \left(\frac{\partial  (h_{\beta}(x_j))  }{\partial x_i} \right) \in \C[G],$$
then~$\mathcal{B}_{t,\gamma}^{(2)}(\beta)$ is equal to~$R_A$.
\end{proposition}

\begin{proof}
Fix a lift of~$z$ to the universal cover $p \colon \widetilde{D}_n \to D_n$. Given a homeomorphism~$H_\beta$ representing a braid~$\beta$, let~$\widetilde{H}_\beta$ be the map induced by the lift of~$H_\beta$ on the chain group~$C_1(\widetilde{D}_n,\widetilde{z})$ (where $\widetilde{z}=p^{-1}(z)$). As~$H_1^{(2)}(D_n,z; \phi, \gamma,t) \cong \ell^2(G) \otimes_{\mathbb{Z}[\pi_1(D_n)]}C_1(\widetilde{D}_n,\widetilde{z})$, it remains to compute the operator $\id \otimes \widetilde{H}_{\beta}$. Clearly~$\widetilde{H}_{\beta}(\widetilde{x}_j)$ is the lift of a loop representing~$h_{\beta}(x_j)$ to the universal cover. Fox calculus then shows that on the chain level
\begin{equation*}
\widetilde{H}_{\beta}(\widetilde{x}_j)   =\sum_{i=1}^n \frac{\partial (h_{\beta}(x_j))}{\partial x_i} \widetilde{x}_i.
\end{equation*}
As we view elements of the left~$\mathbb{Z}[\pi_1(D_n)]$-module~$C_1(\widetilde{D}_n,\widetilde{z})$ as column vectors,~$\widetilde{H}_\beta$ is represented by the~$(n \times n)$ matrix whose~$(i,j)$ component is~$\frac{\partial (h_{\beta}(x_j))}{\partial x_i}$. The claim now follows from the right~$\mathbb{Z}[\pi_1(D_n)]$-module structures of~$\ell^2(G)$.
\end{proof}

\begin{example}
\label{ex:sigma1unred}
 A short computation involving Fox calculus shows that
$
\frac{\partial (h_{\sigma_i}(x_i)  ) }{\partial x_i}=\frac{\partial(x_i x_{i+1} x_i^{-1})}{\partial x_i}\\
 = 1-x_i x_{i+1} x_i^{-1}, 
$
and 
$
\frac{\partial (h_{\sigma_i}(x_i) )}{\partial x_{i+1}}=\frac{\partial(x_i x_{i+1} x_i^{-1}) }{\partial x_{i+1}}= x_i.
$
Consequently, with respect to the good bases, the~$L^2$-Burau maps of~$\sigma_i$  are given by
\[
 \mathcal{B}^{(2)}_{t,\gamma}(\sigma_i)
= \Id^{\oplus (i-1)} \oplus
\begin{pmatrix}
\Id-t R_{\gamma(x_i x_{i+1} x_i^{-1})} & \Id \\
t R_{\gamma(x_i)} & 0
\end{pmatrix}
 \oplus \Id^{\oplus (n-i-1)}.
\]
 \end{example}
 
\begin{example}
Using Proposition \ref{prop:BurauFox}, let us illustrate Lemma \ref{lem:cocycle} with an example.  For $\sigma_1, \sigma_2 \in B_3$, one has
\[
 \mathcal{B}^{(2)}_{t,\gamma}(\sigma_2)
=
\begin{pmatrix}
\Id & 0 & 0 \\
0 & \Id-t R_{\gamma(x_{2} x_{3} x_2^{-1})} & \Id \\
0 & t R_{\gamma(x_{2})} & 0
\end{pmatrix},
 \mathcal{B}^{(2)}_{t,\gamma \circ h_{\sigma_2}}(\sigma_1)
=
\begin{pmatrix}
\Id-t R_{\gamma(x_1 x_{2} x_{3} x_2^{-1} x_1^{-1})} & \Id & 0 \\
t R_{\gamma(x_1)} & 0 & 0 \\
0 & 0 & \Id
\end{pmatrix},
\]
and their composition is equal to
\[
\mathcal{B}^{(2)}_{t,\gamma}(\sigma_2) \circ
\mathcal{B}^{(2)}_{t,\gamma \circ h_{\sigma_2}}(\sigma_1)
=
\begin{pmatrix}
\Id-t R_{\gamma(x_1 x_{2} x_{3} x_2^{-1} x_1^{-1})} & \Id & 0\\
t R_{\gamma(x_1)} - t^2 R_{\gamma(x_1 x_{2} x_{3} x_2^{-1})} & 0 & \Id \\
t^2 R_{\gamma(x_1 x_2)} & 0 & 0 
\end{pmatrix},
\]
which coincides with $ \mathcal{B}^{(2)}_{t,\gamma}(\sigma_1 \sigma_2)$.
 \end{example}

Let us now relate the $L^2$-Burau maps to the classical Burau representation $\mathcal{B}$. Given $\beta \in B_n$, recall that a matrix for $\mathcal{B}(\beta) \in M_{n,n}(\Z[T,T^{-1}])$ can be obtained by computing $T^{\phi}\left (\dfrac{\partial h_{\beta}(x_i)}{\partial x_j}\right )$,
where the ring homomorphism $T^{\phi}\colon \Z[F_n] \to \Z[T,T^{-1}]$ sends $x_i$ to the indeterminate $T$. For any given $\gamma$ and $t$, the $L^2$-Burau map $\mathcal{B}^{(2)}_{t,\gamma}$ holds at least as much information as the classical Burau representation, in the following sense:

\begin{proposition}\label{prop:L2classical}
Given $\beta \in B_n$, for any $t>0$ and $\gamma\colon F_n \to G$, 
there exists a map $\Theta \colon B(\ell^2(G)^n) \to M_{n,n}(\Z[T,T^{-1}])$ such that
$\Theta\left (\mathcal{B}^{(2)}_{t,\gamma}(\beta)\right ) = \mathcal{B}(\beta)$.
In particular, if $\alpha, \beta \in B_n$ and $\mathcal{B}^{(2)}_{t,\gamma}(\alpha) =
\mathcal{B}^{(2)}_{t,\gamma}(\beta)$, then $\mathcal{B}(\alpha)=\mathcal{B}(\beta)$.
\end{proposition}

\begin{proof}
Using Proposition \ref{prop:BurauFox}, $\mathcal{B}^{(2)}_{t,\gamma}(\beta)$ is the right-multiplication operator $R_A$ where the matrix $A$ has $\kappa(F_n,\phi,\gamma,t)
  \left(\frac{\partial  (h_{\beta}(x_j))}{\partial x_i} \right) \in \C[G]$ as its $(i,j)$-coefficient. By considering the map $\theta \colon B(\ell^2(G)^n) \to M_{n,n}(\ell^2(G))$ which evaluates an operator $S$ on the $n$ canonical (column) vectors of $\ell^2(G)^n$, one can extract $A=\theta(R_A)$ from $R_A$. Thus it only remains to recover $\mathcal{B}(\beta)$ from $A$.
  
Since $(\pi_1(D_n),\phi,\gamma)$ is an admissible triple, there exists a homomorphism $\psi\colon G \to \Z$ such that $\phi = \psi \circ \gamma$. Defining the homomorphism $T^\psi\colon G \to \{T^m;m \in \Z\} \subset \Z[T,T^{-1}]$ by $g \mapsto T^{\psi(g)}$, the $(i,j)$-coefficient of $\kappa(G,\psi,T^{\psi},t^{-1})(A)$ is 
  $$ \left (\kappa(G,\psi,T^{\psi},t^{-1}) \circ \kappa(F_n,\phi,\gamma,t)\right )\left(\frac{\partial  (h_{\beta}(x_j))  }{\partial x_i} \right)
  =T^{\phi}\left (\dfrac{\partial h_{\beta}(x_j)}{\partial x_i}\right ),$$
which is precisely the $(j,i)$-coefficient of $\mathcal{B}(\beta)$. 
The map $\Theta = \mathsf{tra} \circ \kappa(G,\psi,T^{\psi},t^{-1}) \circ \theta$ thus satisfies the assumptions of the proposition (where $\mathsf{tra}$ is the transpose operator).
\end{proof}

\begin{remark}
Although all $L^2$-Burau maps recover the Burau representation, different choices of $\gamma \colon F_n \rightarrow G$ produce various effects on the injectivity of the resulting maps and their defect to being anti-representations. On one end of the spectrum, if $\gamma$ is the identity, the~$L^2$-Burau maps~$\mathcal{B}^{(2)}_{t,\id} \colon B_n \to B(\ell^2(F_n)^n)$ are injective for all~$t>0$ (since~$B_n \to Aut(F_n), \beta \mapsto h_{\beta}$ is injective and automorphisms of the free group are determined by their Fox jacobian \cite[Proposition 9.8]{BZ}).  As $G$ becomes smaller, the~$L^2$-Burau maps~$\mathcal{B}^{(2)}_{t,\gamma}$ lose in  injectivity but edge closer to being actual anti-representations. As the proof of Proposition \ref{prop:L2classical} demonstrates, a critical step appears when~$\gamma$ reaches $T^\phi$: in this case, $\mathcal{B}^{(2)}_{t,T^\phi}(\beta)$ is an anti-representation which is equal to $R_{\mathsf{tra}(\mathcal{B}(\beta))}$ up to a change of variable; in particular it is known not to be faithful for $n \geq 5$ \cite{LongPaton, Bigelow}.

Summarizing, the various~$L^2$-Burau maps distinguish at least as many braids as the Burau representation (as shown in Proposition \ref{prop:L2classical}) but sometimes do better as Corollary \ref{cor:faithful} will show. 

%
%
\end{remark}

\subsection{The reduced~$L^2$-Burau map}
\label{sub:reduced}
In this subsection, we shall generalize the definition of the reduced Burau representation to the~$L^2$-setting.
\medbreak
Instead of working with the free generators~$x_1,x_2\dots,x_n$ of~$\pi_1(D_n),$ consider the elements~$g_1, g_2, \dots,g_n,$ where~$g_i=x_1 x_2 \cdots x_i$. The action of the braid group~$B_n$ on this new set of free generators for~$\pi_1(D_n)$ is given by
\[
 h_{\sigma_i}(g_j)=
\begin{cases}
 g_j                            & \mbox{if }  j \neq i, \\ 
g_{i+1}g_{i}^{-1} g_{i-1}  & \mbox{if }  j=i \neq 1, \\
g_2 g_1^{-1}            & \mbox{if } j=i=1. \\ 
 \end{cases} 
\] 
Let~$\widetilde{g}_i$ be the lift of~$g_i$ starting at a fixed lift of~$z$ (note that~$\widetilde{g}_i= \widetilde{x}_1 + \ldots + (x_1 \ldots x_{i-1}) \widetilde{x}_i$). Using the same argument as in Lemma \ref{lem:Relative$L^2$-}, one obtains the splitting 
$$H_1^{(2)}(D_n,z;\phi,\gamma,t)  = \bigoplus_{i=1}^{n-1} \ell^2(G) \widetilde{g}_i \oplus \ell^2(G) \widetilde{g}_n$$
for any $\gamma \colon F_n \to G$ through which $\phi$ factors. As~$g_n$ is always fixed by the action of the braid group, its lift~$\widetilde{g}_n$ is fixed by the lift~$\widetilde{H}_\beta$ of a homeomorphism~$H_\beta$ representing a braid~$\beta$.

\begin{definition}
The \textit{reduced~$L^2$-Burau map} sends a braid~$\beta$ to the restriction~$\overline{\mathcal{B}}_{t,\gamma}^{(2)}(\beta)$ of the~$L^2$-Burau map to the subspace of $H_1(D_n,z;\phi,\gamma \circ h_\beta,t)$ generated by $\widetilde{g}_1,\dots,\widetilde{g}_{n-1}$.
\end{definition} 

The next proposition now follows immediately.

\begin{proposition}
\label{prop:reducedfox}
If~$\widetilde{\mathcal{B}}_{t,\gamma}^{(2)}(\beta)$ denotes the~$L^2$-Burau matrix of a braid~$\beta \in B_n$ with respect to the basis of the~$\widetilde{g}_i$, then 
\[ 
\widetilde{\mathcal{B}}_{t,\gamma}^{(2)}(\beta)=
 \begin{pmatrix}
\overline{\mathcal{B}}_{t,\gamma}^{(2)}(\beta) & 0 \\
V & \Id
\end{pmatrix},   \]
where~$V \in M_{1,n-1}(B(\ell^2(G)))$.
\end{proposition}

One can see the matrix of operators~$\widetilde{\mathcal{B}}_{t,\gamma}^{(2)}(\beta)$ as a conjugate of~${\mathcal{B}}_{t,\gamma}^{(2)}(\beta)$ via a trigonal change of basis matrix between the good basis~$(\widetilde{x}_i)_{1 \leqslant i \leqslant n}$ and the new basis~$(\widetilde{g}_i)_{1 \leqslant i \leqslant n}$. In particular the reduced~$L^2$-Burau map also satisfies the property of Lemma \ref{lem:cocycle} :
$$
\overline{\mathcal{B}}^{(2)}_{t,\gamma}(\alpha\beta)
=
\overline{\mathcal{B}}^{(2)}_{t,\gamma}(\beta) \circ
\overline{\mathcal{B}}^{(2)}_{t,\gamma \circ h_{\beta}}(\alpha).
$$

\begin{example}
\label{ex:sigma1}
Combining Proposition \ref{prop:BurauFox} and Proposition \ref{prop:reducedfox}, the reduced~$L^2$-Burau map of~$\sigma_i \in B_n$ is given by
\begin{align*}
\overline{\mathcal{B}}_{t,\gamma}^{(2)}(\sigma_i)
 &=\Id^{\oplus(i-2)} \oplus 
\begin{pmatrix}
\Id & tR_{\gamma(g_{i+1}g_i^{-1})} & 0 \\
0 & -tR_{\gamma(g_{i+1}g_i^{-1})}  & 0 \\
0 & \Id & \Id
\end{pmatrix}
 \oplus \Id^{\oplus(n-i-2)} 
 \end{align*}
for~$1<i<n-1$, and for~$\sigma_1$ and~$\sigma_{n-1}$ it is represented by
\begin{align*}
\overline{\mathcal{B}}_{t,\gamma}^{(2)}(\sigma_1)
 &=
\begin{pmatrix}
 -tR_{\gamma(g_{2}g_1^{-1})}  & 0 \\
 \Id & \Id
\end{pmatrix}
 \oplus \Id^{\oplus (n-3) }, \\
\overline{\mathcal{B}}_{t,\gamma}^{(2)}(\sigma_{n-1})
 &=\Id^{\oplus (n-3)} \oplus 
\begin{pmatrix}
\Id & tR_{\gamma(g_{n}g_{n-1}^{-1})}  \\
0 & -tR_{\gamma(g_{n}g_{n-1}^{-1})}   \\
\end{pmatrix}.
 \end{align*}
\end{example}

\subsection{Relation to the~$L^2$-Alexander torsions of links}
\label{sub:thm}
In this subsection, we show how a particular $L^2$-Alexander torsion associated to a link can be computed from some reduced $L^2$-Burau maps. As an application, we exhibit two braids which are distinguished by the $L^2$-Burau maps but can not be told apart by the classical Burau representation.
\medbreak
Let~$X_\beta$ be the exterior of a braid~$\beta \in B_n$ in the cylinder~$D^2 \times [0,1].$ The manifold obtained by gluing~$X_\beta$ and~$X_{\xi_n}$ along~$D_n \sqcup D_n$ is nothing but the exterior of the link~$L':=\hat{\beta} \cup \partial D_n$ in~$S^3$. Identify the free group~$F_n$ with~$\pi_1(D_n)$ so that the free generators~$x_i$ correspond to the loops described in Section \ref{sub:braids}. As in Section \ref{sub:reduced}, the elements~$g_1,g_2,\dots,g_n$ then also form a free generating set of~$\pi_1(D_n)$. If~$x$ is a meridian of~$\partial D_n$, then the fiberedness of~$M_{L'}$ implies that~$G_{L'}$ admits the presentation
$$P' = \langle g_1, \ldots, g_n, x |
h_{\beta}(g_1) = x g_1 x^{-1}, \ldots, h_{\beta}(g_n) = x g_n x^{-1} \rangle.$$
The exterior~$M_L$ of~$L=\hat{\beta}$ can now be recovered by canonically pasting a solid torus on the boundary component of~$M_{L'}$ corresponding to~$\partial D_n$. Since~$h_{\beta}(g_n) = g_n$ in the free group~$F_n$,~$G_L$ thus admits the following deficiency one presentation:
$$P = \langle g_1, \ldots, g_n|
h_{\beta}(g_1) = g_1, \ldots, h_{\beta}(g_{n-1}) =  g_{n-1} \rangle.$$
Finally, denote by~$\gamma_L\colon F_n \to G_L$ the resulting quotient map. This way, if one sets~$\phi_L:=(1,\dots, 1) \circ \alpha_L$, then the map~$\phi \colon \pi_1(D_n) \rightarrow \mathbb{Z}$ described in Subsection \ref{sub:Burau} factors as~$\phi_L \circ \gamma_L$.

\begin{theorem}
\label{thm:main}
Given an oriented link~$L$ obtained as the closure of a braid~$\beta \in B_n$, one has
$$
T^{(2)}_{L,(1,\ldots,1)}(\id)(t) \cdot \max(1,t)^{n}
\ \dot{=} \ \det{}^{r}_{\mathcal{N}(G_L)}\left (
\overline{\mathcal{B}}^{(2)}_{t,\gamma_L}(\beta) - \Id^{\oplus (n-1)}
\right )$$
for all~$t > 0$.
\end{theorem}

\begin{proof}
Fix~$t>0$ and assume that~$L$ is non-split. 
Performing Fox calculus on the presentation~$P$ yields 
$$\frac{\partial( h_{\beta}(g_j) g_j^{-1} )}{\partial g_i}=\frac{\partial (h_{\beta}(g_j))}{\partial g_i}- \delta_{ij}.$$
Since~$M_L$ is irreducible and the previously described presentation~$P$ of~$G_L$ has deficiency one, combining Proposition \ref{prop:Fox torsion} 
 with the definition of the reduced Burau representation then gives
$$ T^{(2)}(M_L,\phi_L,\id)(t) \ \dot{=} \
\dfrac{ \det^r_{\mathcal{N}(G_L)}\left (\overline{\mathcal{B}}^{(2)}_{t,\gamma_L}(\beta) - \Id^{\oplus (n-1)} \right ) } 
{ \det^r_{\mathcal{N}(G_L)}(t^n R_{g_n} - \Id)}
= 
\dfrac{ \det^r_{\mathcal{N}(G_L)}\left (\overline{\mathcal{B}}^{(2)}_{t,\gamma_L}(\beta) - \Id^{\oplus (n-1)} \right ) } 
{ \max(1,t)^n}
,$$ 
which proves the theorem in the non-split case.

Next, assume that~$L$ is split. Since Lemma \ref{lem:split} implies that~$T^{(2)}_{L,(1,\ldots,1)}(\id)(t)=0$, it only remains to prove that~$ \det{}^{r}_{\mathcal{N}(G_L)}\left ( \overline{\mathcal{B}}^{(2)}_{t,\gamma_L}(\beta) - \Id^{\oplus (n-1)} \right )$ also vanishes. By Proposition \ref{prop:Fox torsion}, the latter claim reduces to proving that~$C_*^{(2)}(W_{P},\phi_L, \id,t)$ is not weakly acyclic.
As~$C_*^{(2)}(M_L,\phi_L, \id,t)$ is not weakly acyclic (by Lemma \ref{lem:split}), the~$L^2$-version of the Torres formula \cite[Theorem 3.8]{BAthesis} implies that~$C_*^{(2)}(M_{L'},\phi_L\circ Q, Q,t)$ is not weakly acyclic either, where~$Q\colon G_{L'} \to G_{L}$ is the epimorphism induced by the inclusion $M_{L'} \subset M_L$. Since~$L'$ is non split,~$M_{L'}$ is simply homotopy equivalent to~$W_{P'}$  (by Lemma \ref{lem:simplehomot}) and it follows that~$C_*^{(2)}(W_{P'},\phi_L\circ Q, Q,t)$ is not weakly acyclic, by \cite[Theorem 2.12]{BAthesis}.
Let~$v$ be the~$0$-cell of~$W_P$,~$g_1,\dots, g_n$ be its~$1$-cells and~$r_1,\dots,r_{n-1}$ be its 2-cells. Similarly let~$v'$ be the~$0$-cell of~$W_{P'}$,~$g_1,\dots, g_n,x$ be its~$1$-cells and~$r_1',\dots,r_n'$ be its 2-cells. Denote the lifts to the universal covers as in Section \ref{sub:L2TorsionFox} and set~$D_1= \ell^2(G) \widetilde{x}$,~$D_2= \ell^2(G) \widetilde{r}_n'$. A straightforward matrix computation involving Fox calculus now shows that
$$0 \to C_*^{(2)}(W_{P},\phi_L, \id,t) \stackrel{\iota}\to 
C_*^{(2)}(W_{P'},\phi_L\circ Q, Q,t) \stackrel{\rho} \to D_* \to 0$$
is an exact sequence of finite Hilbert~$\NN(G_L)$-chain complexes, where~$\iota_1 (\widetilde{g}_i)=\widetilde{g}_i', \iota_2 (\widetilde{r}_i)=\widetilde{r}_i'$ for~$i=1,\dots, n-1$ and~$\rho_1,\rho_2$ are the obvious projections.
As the boundary operator~$D_2 \rightarrow D_1$ is given by the injective operator~$\Id - t^n R_{g_n}$, the chain complex~$D_*$ is weakly acyclic. Since~$C^{(2)}(W_{P'},\phi_L\circ Q, Q,t)$ is not weakly acyclic, neither is~$C^{(2)}(W_{P},\phi_L, \id,t)$ by Proposition \ref{prop:exact}. This concludes the proof.
\end{proof}

\begin{figure}[h]
\centering
\begin{tikzpicture}

\begin{scope}[scale=0.3]

\braid[
 style all floors={fill=yellow},
 style floors={1}{dashed,fill=yellow!50!green},
 floor command={%
 \draw (\floorsx,\floorsy) -- (\floorex,\floorsy);
 },
 line width=1pt,
 style strands={1}{red},
 style strands={2}{blue},
  style strands={3}{black},
   style strands={4}{purple},
    style strands={5}{brown},
 style strands={6}{green}
] 

(braid) at (2,0) 
s_1^{-1}-s_5  s_2-s_4^{-1} s_3  s_4-s_2^{-1}  s_5^{-1}-s_1
s_5^{-1}-s_1  s_5^{-1}-s_1  s_4-s_2^{-1}  s_3   s_2-s_4^{-1}
 s_1^{-1}-s_5  s_1^{-1}-s_5 
s_1^{-1}-s_5  s_2-s_4^{-1} s_3^{-1}  s_4-s_2^{-1}  s_5^{-1}-s_1
s_5^{-1}-s_1  s_5^{-1}-s_1  s_4-s_2^{-1}  s_3^{-1}   s_2-s_4^{-1}
s_1^{-1}-s_5  s_1^{-1}-s_5 
;

\end{scope}
\end{tikzpicture}
\caption{The braid $\beta \in B_6$. }
\label{fig:braid44}
\end{figure}
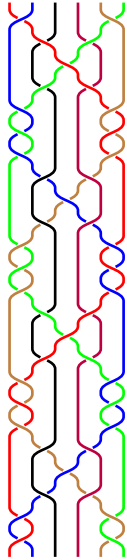

\begin{remark}
If~$L$ is a knot~$K$, then Theorem \ref{thm:main} can be expressed as
$$
\Delta^{(2)}_{K}(t) \cdot \max(1,t)^{n-1} \ \dot{=} \ \det{}^{r}_{\mathcal{N}(G_K)}\left ( \overline{\mathcal{B}}^{(2)}_{t,\gamma_K}(\beta) - \Id^{\oplus (n-1)} \right ),$$
where $\Delta^{(2)}_{K}(t)$ is the~$L^2$-Alexander invariant of $K$ defined by Li-Zhang \cite{LZ06a}.
\end{remark}

\begin{corollary}\label{cor:faithful}
There exist two braids which have the same image under the classical Burau representation but have different images under an~$L^2$-Burau map.
\end{corollary}

\begin{proof}
Long and Paton \cite{LongPaton} proved that the braid~$\beta \in B_6$ depicted in Figure \ref{fig:braid44} has the same image under the classical Burau representation as the trivial braid~$\xi_6 \in B_6$. In order to prove that~$\mathcal{B}^{(2)}_{t,\gamma_{\hat{\beta}}}(\beta) \neq \mathcal{B}^{(2)}_{t,\gamma_{\hat{\beta}}} (\xi_6)$, we will show that~$\overline{\mathcal{B}}^{(2)}_{t,\gamma_{\hat{\beta}}}(\beta) \neq \overline{\mathcal{B}}^{(2)}_{t,\gamma_{\hat{\beta}}} (\xi_6)$: this is enough since the reduced~$L^2$-Burau map is the upper left matricial part of the~$L^2$-Burau map expressed in the basis of the $\widetilde{g}_i$.
We claim that the closure $L$ of $\beta$ is a $6$-component non-split link.
To see this, define $\Gamma(L)$ to be the graph whose vertices are the components $L_i$ of $L$ and such that there is an edge between $L_i$ and $L_j$ when there exists a third component $L_k$ such that $L_i \cup L_j \cup L_k$ is a non-split link. Since $L$ being split implies $\Gamma(L)$ being disconnected, it suffices to show that $\Gamma(L)$ is connected.
One can observe that all sublinks of L with three components are either trivial or the non-split link $L_{10a140}$, and
 there are enough of the second type so that $\Gamma(L)$ is connected.

Consequently, as $L$ is non-split,~$T^{(2)}_{L,(1,\ldots,1)}(t)$ is non-zero for all~$t>0$ (by Lemma \ref{lem:split}) and thus Theorem \ref{thm:main} implies that the operator~$\overline{\mathcal{B}}^{(2)}_{t,\gamma_L}(\beta) - \Id^{\oplus (n-1)}$ has non-zero regular Fuglede-Kadison determinant and is thus injective. The result follows immediately.
\end{proof}

\bibliographystyle{plain}
\bibliography{biblio}

\end{document}